\documentclass[10pt,twoside]{siamart1116}

\usepackage[english]{babel}
\usepackage{graphicx,epstopdf,epsfig}
\usepackage{amsfonts,epsfig,fancyhdr,graphics, hyperref,amsmath,amssymb}
\usepackage{color}

\newcommand{\HE}{Namie of Handling Editor}
\newcommand{\DoS}{Month/Day/Year}
\newcommand{\DoA}{Month/Day/Year}
\newcommand{\CA}{Vu Le Anh}

\newcommand{\Names}{Hieu Ha Van, Le Anh Vu and Hoa Duong Quang}
\newcommand{\Title}{Classification of a Class of Solvable Real Lie Algebras by using Techniques in Matrix Theory}

\newtheorem{remark}[theorem]{Remark}

\newcommand{\reals}{\mathbb{R}}
\newcommand{\complex}{\mathbb{C}}
\newcommand{\Gc}{\mathcal G}
\newcommand{\Hc}{\mathcal H}
\DeclareMathOperator*{\ima}{Im}
\DeclareMathOperator*{\rank}{rank}
\newcommand{\MD}[2]{\text{MD}_{#2}(#1)}
\newcommand{\ad}[1]{\textnormal{ad}_{#1}}
\newcommand{\adg}[1]{\textnormal{ad}^1_{#1}}
\newcommand{\adgg}[1]{\textnormal{ad}^2_{#1}}
 
\renewcommand{\theequation}{\thesection.\arabic{equation}}
\renewtheorem{theorem}{Theorem}[section]

\begin{document}

\bibliographystyle{plain}

\setcounter{page}{1}

\thispagestyle{empty}

 \title{\Title\thanks{Received
 by the editors on \DoS.
 Accepted for publication on \DoA. 
 Handling Editor: \HE. Corresponding Author: \CA}}

\author{
    Hieu Ha Van\thanks{University of Economics and Law, Vietnam National University Ho Chi Minh City, Vietnam
    (hieuhv@uel.edu.vn, \newline 
    vula@uel.edu.vn). This research is funded by Vietnam National University Ho Chi Minh City (VNU-HCM), under grant number B2022-34-05.}
    \and
    Vu Le Anh\footnotemark[2]
    \and
    Hoa Duong Quang\thanks{Hoa Sen University, Vietnam (hoaduongquang@hoasen.edu.vn).}
}

\markboth{\Names}{\Title}

\maketitle

\begin{abstract}
We give a complete classification of the class of Lie algebras of simply connected real Lie groups whose nontrivial coadjoint orbits are of codimension 1. Such a Lie group belongs to a well-known class, called the class of MD-groups. The Lie algebra of an MD-group is called an MD-algebra. Some interest properties of MD-algebras will be investigated as well. The techniques used in this paper is elementary techniques in matrix theory and available to apply to more general cases.
\end{abstract}


\section{Introduction}
    
    As we know, an $n$-dimensional Lie algebra can be defined via the commutators of its basis vectors. More precisely, in order to define an $n$-dimensional Lie structure $P$ on an $n$-dimensional vector space $\Gc$ spanned by $\{x_1,x_2,\dots,x_n\}$, we need to define the Lie brackets 
    \begin{equation*}
        [x_i,x_j]=\sum\limits_{k=1}^n a_{ij}^kx_k \quad (1\leq i,j\leq n)
    \end{equation*}
    which satisfy two properties:
    \begin{align}
        [x_i,x_j]=-[x_j,x_i] & \quad \forall i,j \tag*{(skew-symmetrix)} \\ 
        \left[x_i,[x_j,x_k]\right] +\left[x_j,[x_k,x_i]\right]+\left[x_k,[x_i,x_j]\right] = 0 & \quad \forall i,j,k. \tag*{(Jacobi identity)}
    \end{align}
    In other words, we need to consider the family $\{\ad{x_i}:i=1,2,..,n\}$ of $n$ adjoint operators on $\Gc$ which satisfies the following requirements:
     \begin{align}
        \ad{x_i}(x_j)=-\ad{x_j}(x_i) & \quad \forall i,j \tag*{(skew-symmetrix)} \\ 
        \ad{x_i}\circ\ad{x_j} -\ad{x_j}\circ\ad{x_i} = \ad{[x_i,x_j]} & \quad \forall i,j. \tag*{(Jacobi identity)}
    \end{align}
    Let's denote by $A_i$ the $n \times n$ matrix  of the adjoint operator $\ad{x_i}$ with respect to the basis $\{x_1, \ldots, x_n\}$, $i = 1, \ldots, n$. Note that the family of matrices $\mathcal{A}:=(A_1,\ldots,A_n)$ is determined by the $n$-dimensional Lie structure $P$. 
    
    According to Kirillov et al. \cite{KN87}, the set $\mathcal{P}$ of all $n$-dimensional Lie structures $P$ on $\Gc$ determines the set $\mathcal{A}_n : = \{\mathcal{A} : P \in \mathcal{P}\} \subset \Gc \otimes \Gc^* \otimes \Gc^*$. Due to the skew-symmetry and Jacobi identity, the set $\mathcal{A}_n$ forms an algebraic veriety in $n^3$-dimensional space $\Gc \otimes \Gc^* \otimes \Gc^*$ and it is called the variety of $n$-dimensional Lie algebra structures in term of A. A. Kirillov et al. \cite{KN87}.
   
   Note that, by Kirillov et al. \cite{KN87}, the natural action of the general linear group $GL(\Gc) \equiv GL(n, \reals)$  on $\Gc \otimes \Gc ^{*} \otimes \Gc ^{*}$ takes the variety $\mathcal{A}_n$ into itself and the orbits of $GL(n, \reals)$ in $\mathcal{A}_n$ correspond to the isomorphic classes of $n$-dimensional Lie algebras. In other words, the orbit space $\mathcal{A}_n / GL(n, \reals)$ coincides with the set of isomorphic classes of $n$-dimensional real Lie algebra. Hence, the number of parameters increases drastically and the volume of calculations, therefore, will become enormous when $n$ not small \cite{Boza13}. That explains why the problem of classification of Lie algebras (as well as Lie groups) of dimension $n$ strictly greater than 6 is still open, in general, until now. In some certain cases, the problem of the classification is proven to be ``will" and ``hopeless" \cite{FKPS18,FKPS18-2,FKPS18-3,FKPS18-4}.
    
    However, it is still possible to classify some interest subclass of solvable Lie algebras (of general finite dimensions) added one or some special properties in which we can use techniques in matrix theory to make the matrices in the family $\mathcal{A} = (A_1, \ldots, A_n)$ to be as simple as possible. This paper will investigate such a class, called MD-algebras. Basically, we know that Lie groups whose coadjoint orbits are all trivial (i.e. of dimension 0) are trivial. But, in the case they have non-trivial coadjoint orbits, the problem of classification is not easy and still open. The basis idea of MD-groups (as well as MD-algebras) is to consider Lie groups (and their Lie algebras) in which all non-trivial coadjoint orbits are of the same dimension. This idea comes from the Kirillov's method \cite{Kiri} and propsed by Do Ngoc Diep while searching for the class of Lie groups whose $C^*$-algebra can be characterized by BDF K-functions in 1980 \cite{Do99}. More precisely, Diep recommended to consider finite-dimensional solvable and simply connected Lie groups whose $K$-orbits are either of dimension 0 or of a constant positive dimension (which is called the maximal dimension). This Lie group is said to have the property MD and is called an MD-group. The Lie algebra of an MD-group is called a Lie algebra having the property MD or an MD-algebra for brevity. 
	It is worth pointing out that the family of $K$-orbits of the maximal dimension of an MD-group forms a measured foliation in terms of A. Connes \cite{Le90-2,Con82}. 
	
	The problem of classifying MD-algebras has received much attention in recent decades, but the results are not much.
	\begin{itemize}
	    \item In 1984, Vuong Manh Son and Ho Huu Viet gave a complete classification of MD-groups whose maximal dimension of $K$-orbits is equal to the dimension of the group \cite{Son-Viet}.
	    \item In 1990, Vu L. A. (the second author of this paper) gave the classification (up to isomorphism) of  all MD-algebras of dimension 4 \cite{Le90-2, Le93}. 
	    \item In 1995, D. Arnal et al. gave a list of all MD-algebras in which the maximal dimension of $K$-orbits of corresponding MD-groups is 2 \cite{Arnal}. 
	    \item In 2012, Vu L. A., Hieu H. V. (the first author of this paper) et al. gave the classification (up to isomorphism) of MD-algebras of dimension 5 \cite{Vu-Shum, LHT11}.
	    \item In 2016, Vu L. A., Hieu H. V. et al. classified (up to isomorphism) all of MD-algebras which have the first derived ideal of codimension 1 or 2 \cite{Vu-MD}.
	    \item In 2019, M. Goze et al. gave the classification (up to isomorphism) of all of MD-algebras such that their nontrivial $K$-orbits are of dimension 4 \cite{GR19}.
	\end{itemize}
	
	Basically, the approaches of the classifying MD-algebras obtained in the above papers can be divided into the following directions: (1) fixing the dimensions of Lie algebras and (or) the dimensions of the derived series \cite{Le90-2, Le93, Vu-Shum, LHT11, Vu-MD}, (2) fixing the maximal dimension of K-orbits \cite{Arnal, GR19}, and (3) fixing the codimension of K-orbits \cite{Son-Viet}. Our paper follows the third direction. Specifically, we will classify  all MD-algebras whose nontrivial $K$-orbits are of codimension 1 (i.e. the maximal dimension of $K$-orbits is 1 less than the dimension of such a Lie algebra).
 
    The paper is organized into three sections, including this introduction. In Section 2 we recall some basic preliminary concepts, notations and properties which will be used throughout the paper. The last section will be devoted to setting and proving the main results of the paper. 

\section{\bf Preliminaries}
    Throughout this paper, the following notations will be used.
    \begin{itemize}
    	\item An $n\times n$ matrix whose $(i,j)$-entry is $a_{ij}$ will be written as $(a_{ij})_{n\times n}$. While the $(i,j)$-entry of a matrix $A$ will be denoted by $(A)_{ij}$. To simplify notation we use the same letter $A$ for the endomorphism of $\mathbb{F}^n$ which is defined by assigning a column vector $v\in\mathbb{F}^n$ to $Av$, assumed that the entries of $A$ take values in $\mathbb{F}$. Note that $\ima A$ is equal to the $\mathbb{F}$-vector space spanned by the columns of $A$. We shall denote by $\boldsymbol{0}$ the zero matrix of suitable size.
    	\item With the notation $V=\langle v_1,v_2,\dots,v_n\rangle$, we mean $v_1,v_2,\dots,v_n$ is a basis of the vector space $V$.
    	The dual space of $V$ will be denoted by $V^*$. It is well-know that if $V=\langle v_1,\dots,v_n\rangle$ then $V^*=\langle v_1^*,\dots,v_n^*\rangle$ where each $v_i^*$ is defined by $v_i^*(v_j)=\delta_{ij}$ (the Kronecker delta symbol)  for $1 \leq i, j \leq n$.
    	\item By $\Gc$ we always mean a solvable Lie algebra of finite dimension over the real field with Lie bracket $[\cdot,\cdot]$. 
    	For any $x \in \Gc$, we will denote by $\ad{x}$ the adjoint action of $x$ on $\Gc$, i.e. $\ad{x}$ is the endomorphism of $\Gc$ defined by $\ad{x}(y)= [x, y]$ for every $y \in \Gc$.  
    	By $\adg{x}$, $\adgg{x}$ we mean the restricted maps of $\ad{x}$ on the first derived ideal $\Gc^1:=[\Gc,\Gc]$ and the second derived ideal $\Gc^2:=[\Gc^1,\Gc^1]$, respectively. Since both $\Gc^1$ and $\Gc^2$ are ideals of $\Gc$, we shall consider $\adg{x}$ and $\adgg{x}$ as endomorphisms of $\Gc^1$ and $\Gc^2$, respectively. Similary, we will denote by $F^1$ the restriction $F|_{\Gc^1}$ of any form $F \in \Gc^*$ on $\Gc^1$.	
    \end{itemize}
    Next, we will recall some definitions, elementary results about coadjoint representation of solvable Lie groups. For more details, we refer reader to \cite{Kiri}.

    \begin{definition}
        {\rm Let $G$ be a Lie group and let $\Gc$ be its Lie algebra. If $\text{Ad}:G \rightarrow  \text{Aut}(\Gc)$ denotes the adjoint representation of $G$. Then the action
        \[ 
            \begin{array}{lrll}
                 K: & G & \rightarrow & \text{Aut}(\Gc^*) \\
                    & g & \mapsto & K_g
            \end{array}
        \]
        defined by
        \begin{equation*}
            K_g(F)(x)=F\bigl(\text{Ad}(g^{-1})(x)\bigr); \text{ for } F\in\Gc^*, x\in\Gc
        \end{equation*}
        is called the {\em coadjoint representation} of $G$ in $\Gc^*$.
        Each orbit in the coadjoint representation of $G$ is called a {\em coadjoint orbit} (or a {\em $K$-orbit} for short) of $G$ as well as of $\Gc$. }
    \end{definition}
    
    For each $F\in\Gc^*$, the $K$-orbit that passes through the point $F$ is denoted by $\Omega_F$, i.e. 
    \begin{equation*}
        \Omega_F=\{K_g(F): g\in G\}.
    \end{equation*}

	\begin{definition}\cite[Section 15.1]{Kiri} \label{bf}
	{\rm Let $F$ be any element in $\Gc^*$. The Kirillov's bilinear form with respect to $F$ is defined by
		\[ B_F (x, y) : = F([x, y]); \text{ for } x, y \in \Gc.\]
	It is a bilinear skew symmetric form on $\Gc$.}
	\end{definition}

The matrix of $B_F$ with respect to a basis $\{x_1,x_2,\dots,x_n\}$ of $\Gc$, which is equal to $\bigl(F([x_i,x_j])\bigr)_{n\times n}$, is called the Kirillov's matrix of $F$ with respect to that basis. From now on, when a basis of $\Gc$ is fixed, we will always treat $B_F$ as the Kirillov's matrix of $F$ with respect to the fixed matrix. The set of matrices $\{B_F:F\in\Gc^*\}$ will be called the Kirillov's matrices of $\Gc$. The following proposition gives us a connection between dimension of $\Omega_F$ and the rank of $B_F$.

    \begin{proposition}\cite[Section 15.3]{Kiri} \label{rankbf} Let $F$ be any element in $\Gc^*$. Then 
        \begin{equation*}
            \dim \Omega_F=\rank{B_F}.
        \end{equation*}
    \end{proposition}
    
    As a consequence, $\dim\Omega_F$ is always even for every $F$ in $\Gc^*$ and $\dim\Omega_F>0$ if and only if $F^1$ is not equal to zero, i.e. $F|_{\Gc^1} \neq 0$. It is trivial to see that if $\dim\Omega_F=0$ for every $F\in\Gc^*$ then $\Gc$ is commutative. Our concern will be Lie algebras in which the dimensions of non-trivial $K$-orbits are equal.

    \begin{definition}\cite[Introduction]{Son-Viet}\label{definitionMD} {\rm A 
    finite dimensional solvable and simply connected Lie group is called an {\em MD-group} (in terms of Do Ngoc Diep) if its $K$-orbits are either of dimension zero (trivial ones) or of the same positive constant dimension.  The Lie algebra of an MD-group is called an {\em MD-algebra}. }
    \end{definition}
    
    Let $G$ be an MD-group. Then the positive constant dimension of $K$-orbits is also called the maximal dimension of $K$-orbits of $G$ (as well as of the corresponding Lie algebra $\Gc$). If this maximal dimension is given by $k$ ($k> 0$), we also say that $\Gc$ is an MD$_k$-algebra. Furthermore, if $\dim\Gc=n$ then $\Gc$ is also called an MD$_k(n)$-algebra. In particular, if an MD-group $G$ satisfies one of the following properties: either all $K$-orbits are trivial or the maximal dimension of $K$-orbits is equal to the dimension of $G$ then $G$ is said to be an SMD-group. The Lie algebra of an SMD-group is called an SMD-algebra. It is clear that any commutative Lie algebra is an SMD-algebra and it has no maximal dimension of $K$-orbits because its $K$-orbits are all trivial. 
    
    In order to classify MD$_{n-1}(n)$-algebras, we firstly note that all solvable Lie algebras of dimension $n < 4$ are obviously MD-algebras and can be classified easily. Therefore we only take interest in MD-algebras of dimension $n \geq 4$.
    On the other hand, in 1984 Vuong Manh Son and Ho Huu Viet gave a bound for the solvable index of MD-algebras which is presented in the following proposition.

    \begin{proposition}\cite[Theorem 4]{Son-Viet}\label{nescodiMD}
        If $\Gc$ is an MD-algebra then its second derived ideal $\Gc^2$ is commutative. 
    \end{proposition}
    
    Therefore, if $\Gc$ is an MD-algebra then the third derived ideal $\Gc^3:=[\Gc^2,\Gc^2]$ is the trivial vector space. In this paper, a solvable Lie algebra is said to be {\it $i$-step solvable} if $\Gc^i$ is non trivial (i.e. $\Gc^i\neq \{0\}$) and commutative. By using this term, Proposition \ref{nescodiMD} means that any non-commutative MD-algebra is either 1-step or 2-step solvable. The basis idea of the classification we use in this paper is to classify 1-step solvable MD$_{n-1}(n)$-algebras firstly and to connect 2-step solvable MD$_{n-1}(n)$-algebras with 1-step solvable MD$_{n-1}(n-1)$-algebras afterwards. To deal with 1-step solvable ones, it is necessary to recall the following classical result.

	\begin{proposition}\cite[Lemma 3.1]{Vu-Shum}\label{commutative}
	If $\Gc$ is 1-step solvable then 
	\begin{equation*}
    	\adg{x} \circ \adg{y}=\adg{y}\circ \adg{x} \quad \forall x,y \in\Gc.
    \end{equation*}
	In other words, $\{\adg{x}:x\in \Gc\}$ is a commuting family of linear endomorphisms of $\Gc^1$.
	\end{proposition}
	
    Hence, it is reasonable to consider the set of commuting matrices over the real field. It is well-known that an arbitrary set of commuting  matrices over the complex field $\complex$ has an invariant subspace of dimension 1. Therefore, it may be simultaneously brought to triangular form by a unitary similarity \cite{Morris,Heydar}. Similarly for the triangular form of an arbitrary set of commuting real matrices. We state these classical results in the following Proposition.
    
    \begin{proposition}\label{triangular}
        Let $\mathcal{S}\subseteq M_{n\times n}(\mathbb{F})$ be a set of commuting $n\times n$ matrices over a field $\mathbb{F}$, i.e. $AB=BA$ for every $A,B\in\mathcal{S}$. 
        \begin{itemize}
        	\item[(i)] If $\mathbb{F}=\complex$ then $\mathcal{S}$ is simultaneously triangularizable by a unitary matrix, i.e there exists a complex unitary matrix $T$ so that $T^{-1}AT$ is upper triangular matrix for all $A \in \mathcal{S}$. 
        	\item[(ii)] if $\mathbb{F}=\reals$ then $\mathcal{S}$ is block simultaneously triangularizable by a orthogonal matrix, i.e. there is a real orthogonal matrix $T$ so that $T^{-1}AT$ is a block upper triangular matrix, where each diagonal block is of size either 1 or 2, for all $A \in \mathcal{S}$. 
        \end{itemize} 
    \end{proposition}

\section{\bf Main Results}

In this section, we firstly give some new interest properties of MD$_k(n)$-algebras, especially 1-step solvable MD-algebras [Theorem \ref{theorem0}, \ref{theorem1} \& \ref{theorem2}]. Secondly, we give a connection between 1-step solvable MD-algebras and 2-step solvable MD-algebras [Theorem \ref{theorem3}]. Finally, we will present the complete classification of $\MD{n}{n-1}$-algebras in the last theorem. It is possible that we can apply Theorem \ref{theorem1} and Theorem \ref{theorem3} to obtain the classification of MD$_{k}(n)$-algebras with $n-k$ small (not necessary to be 1) but we will not develop this point here. 

\subsection{Some new characteristics of MD-algebras} 

\begin{theorem}\label{theorem0}
	Let $\Gc$ be a non-commutative Lie algebra. Then $\Gc$ is a decomposable MD$_k$-algebra if and only if it is a trivial extension of an indecomposable MD$_k$-algebra by a finite-dimensional commutative Lie algebra. 
\end{theorem}

\begin{proof}
Let's first prove the ``if" part. Assume that $\Gc$ is a direct sum of an MD$_k$-algebra with a finite-dimensional commutative algebra $\reals^r$, i.e. $\Gc=\mathcal{K}\oplus \reals^r$, where $\mathcal{K}$ is an MD$_k$-algebra. If so, we have
\begin{equation}
    \Gc^1=\mathcal{K}^1.
\end{equation}
Denote by $n$ the dimension of $\Gc$. Let's fix a basis $\{x_1,x_2,\dots,x_n\}$ of $\Gc$ so that $\mathcal{K}=\langle x_1,\dots,x_{n-r}\rangle$. For any $F = F_{\mathcal{K}} + F_r \in \Gc ^*$ with $F_{\mathcal{K}} \in \mathcal{K}^*\subseteq \Gc^*$ and $F_r \in (\reals^r)^*\subseteq \Gc^*$, we denote by ${\bar{\Omega}}_{F_{\mathcal{K}}}$ the $K$-orbit of $F_{\mathcal{K}}$ in $\mathcal{K} ^* \subset \Gc ^*$. It is elementary to see that
\begin{equation}
	B_F=\begin{bmatrix}
		{\bar{B}}_{F_{\mathcal{K}}} & \boldsymbol{0}\\
		\boldsymbol{0} & \boldsymbol{0}
	\end{bmatrix}.
\end{equation}
where ${\bar{B}}_{F_{\mathcal{K}}}$ is the Kirillov's matrix of $F_{\mathcal{K}}$ in $\mathcal{K}$ with respect to the basis $\{x_1,\dots,x_{n-r}\}$. Therefore, we get $\rank B_F = \rank {\bar{B}}_{F_{\mathcal{K}}}$. 
Since $\mathcal{K}$ is an MD$_k$-algebra, we have 
\begin{equation}
    \dim{\Omega_{F}}= \dim {\bar{\Omega}}_{F_{\mathcal{K}}} = \left\{
        \begin{array}{ll}
            0 & \text{if } F|_{\mathcal{K}^1} = \boldsymbol{0} \\
            k & \text{elsewhere}
        \end{array}
    \right.
\end{equation}
This proves $\Gc$ is an MD$_k$-algebra, as required. 

It remains to prove the ``only if'' part. By contradiction, assume that $\Gc = \mathcal{K} \oplus \mathcal{L}$ in which both $\mathcal{K}$ and $\mathcal{L}$ are non-commutative sub-algebras of $\Gc$. Let's denote by $n$ and $m$ the dimensions of $\Gc$ and $\mathcal{K}$, respectively. Since $\Gc = \mathcal{K} \oplus \mathcal{L}$, there is a basis $\{x_1,x_2,\dots,x_n\}$ of $\Gc$ so that
\begin{equation}
    \left\{\begin{array}{l}
        \mathcal{K} = \langle x_1,x_2,\dots,x_m\rangle, \\
        \mathcal{L} = \langle x_{m+1},\dots,x_n\rangle.
    \end{array}\right.
\end{equation}
Because both $\mathcal{K}$ and $\mathcal{L}$ are non-commutative, there exist two non-zero elements $F_1,F_2\in(\Gc)^*$ so that
\begin{equation}
    \left\{
        \begin{array}{ll}
            F_1|_{\mathcal{K}^1} \neq \boldsymbol{0}, & F_1|_{\mathcal{L}^1} = \boldsymbol{0},  \\
            F_2|_{\mathcal{K}^1} = \boldsymbol{0}, & F_2|_{\mathcal{L}^1} \neq \boldsymbol{0}.  
        \end{array}
    \right.
\end{equation}
If so, it is elementary to see that
\begin{equation}
    B_{F_1}  
	=\left[\begin{array}{cl}
		M & \boldsymbol{0}\\
	    \boldsymbol{0} & \boldsymbol{0}
	\end{array}\right], \text{ and }
		B_{F_2}  
	=\left[\begin{array}{cl}
		\boldsymbol{0} & \boldsymbol{0}\\
		\boldsymbol{0} & N
	\end{array}\right],
\end{equation}
where $M$ and $N$ are two non-zero matrices of dimension $m\times m$ and $(n-m)\times (n-m)$, respectively. We thus get
\begin{equation*}
    0<\rank{B_{F_1}}<\rank{B_{F_1}}+\rank{B_{F_2}}=\rank{B_{F_1+F_2}}.
\end{equation*}
In other words, 
\begin{equation*}
    0<\dim \Omega_{F_1}<\dim \Omega_{F_1+F_2},
\end{equation*}
a contradiction to the fact that $\Gc$ is an MD-algebra. This completes the proof.
\end{proof}
\begin{theorem}\label{theorem1}
	Let $\Gc$ be a 1-step solvable  $\MD{n}{k}$-algebra so that $\dim\Gc^1 > n-k$. Then we have the following assertions.
	\begin{enumerate}
		\item $\sum\limits_{y\in \Gc} \ima \adg{y} = \Gc^1.$
		\item There is $x$ in $\Gc$ so that $\adg{x}$ is an automorphism on the vector space $\Gc^1$.
		\item $\Gc$ can be expressed as the semidirect sum $\mathcal{L} \oplus_{\rho} \Gc^1$ of a commutative subalgebra $\mathcal{L}$ of $\Gc$ with the first derived ideal $\Gc^1$ by the representation $\rho: \mathcal{L} \rightarrow Der(\Gc^1)$, which is defined by $\rho(x) = \adg{x}$\, 
		for every $x \in \mathcal{L}$.
	\end{enumerate}
\end{theorem}   

In order to prove this Theorem, we will need the following lemma.

    \begin{lemma}\label{lemma32}
        Let $\mathcal{S} =\{A_1, A_2,\dots, A_s\}$ be a set of complex upper triangular matrices of size $m\times m$. If $\mathcal{S}$ is commutative (i.e. $AB=BA$ for every $A, B\in \mathcal{S}$) and 
        \begin{equation}\label{sum-m}
            \sum\limits_{i=1}^s \ima A_i = \complex^m
        \end{equation}
        then there is a linear combination of elements of $\mathcal{S}$ so that it is nonsingular.
    \end{lemma}
    
    \begin{proof}
        The proof is by induction on $m$. It is trivial for $m=1$. Since $A_i$ is of upper triangular form for every $i$, the equality {\rm (\ref{sum-m})} implies that at least one of $\{(A_i)_{mm}:i=1, \dots, s\}$ is non-zero (if not, the vector $(0,0,\dots,1)^t\notin \sum\limits_{i=1}^s \ima A_i = \complex^m$, a contradiction). By re-indexing if necessary, we may assume that $(A_1)_{mm} \neq 0$. Furthermore, 
        the proof is still correct if we replace $\mathcal{S}$ by $\mathcal{S}'=\{A_1,A_2-\beta_2A_1,\dots,A_s-\beta_sA_1\}$ for every $\beta_2, \dots,\beta_s\in\complex$. Therefore, we may assume from beginning that 
        \begin{equation}\label{sum-m2}
            \left\{\begin{array}{ll}
                (A_i)_{mm}=0 & \forall i\geq 2, \\
                (A_1)_{mm}\neq 0.
            \end{array}\right.
        \end{equation}
        Since $A_1A_i=A_iA_1$ for every $i$, we have
        \begin{equation}
            (A_1A_i)_{jm}=(A_iA_1)_{jm} \text{ for every } i,j.
        \end{equation}
        Therefore, for every $i$, the last column of $A_i$  is a linear combination of the following columns: all the columns of $A_1$ and the first $(m-1)$ columns of $A_i$. 
        It follows that $\sum\limits_{i=1}^s \ima A_i$ is spanned by the last column of $A_1$ and the first $m-1$ columns of $A_1,\dots,A_s$. We conclude from the equality {\rm (\ref{sum-m})} that $\complex^{m-1}$ is spanned by the first $(m-1)$ columns of $A_1, \dots, A_s$. 
        
        In other words, if we denote by $B_i$ the $(m-1)\times (m-1)$ matrix obtained by deleting the $m$-th column and $m$-th row of $A_i$ for each $1\leq i\leq s$ then $\{B_1,B_2,\dots,B_s\}$ satisfies the following properties:
        \begin{itemize}
            \item it is a set of commuting (complex) matrices of dimension $(m-1)\times (m-1)$,
            \item each $B_i$ is of upper triangular form for every $i$, and
            \item $\complex^{m-1}$ is spanned by the columns of $B_1, B_2,\dots,B_s$, i.e.
                \begin{equation}
                    \sum\limits_{i=1}^s \ima B_i = \complex^{m-1}.
                \end{equation}
        \end{itemize}
        By induction, there are $\alpha_1,\alpha_2,\dots,\alpha_s \in \complex$ so that $\alpha_1B_1+\cdots+\alpha_sB_s$ is nonsingular. Remark that we can always choose $\alpha_1 \neq 0$ (if $\alpha_1=0$ then $\alpha_2B_2+\cdots+\alpha_sB_s$ is nonsingular and hence, there exists $\alpha_1'\neq 0$ so that $\alpha'_1B_1+\alpha_2B_2+\cdots+\alpha_sB_s$ is nonsingular because the equation $\det\bigl(B_1+x(\alpha_2B_2+\cdots+\alpha_sB_s)\bigr)=0$ has finite solutions). With the assumption $\alpha_1\neq 0$ in hand, it follows immediately from the equation {\rm (\ref{sum-m2})} and from the nonsingularity of $\alpha_1B_1+\cdots+\alpha_sB_s$ that
        \begin{equation*}
            \alpha_1A_1+\cdots+\alpha_sA_s
        \end{equation*}
       is nonsingular, which proves the lemma.
    \end{proof}
    
    \begin{proof}[Proof of Theorem \ref{theorem1}] Set $m=\dim \Gc^1$. Throughout this proof, we fix a basis $\{x_1,x_2,\dots,x_n\}$ of $\Gc$ so that $\Gc^1=\langle x_{n-m+1}, x_{n-m+2},\dots,x_n\rangle$.
        \begin{enumerate}
            \item By contradiction, assume that $\sum\limits_{y\in\Gc}\ima \adg{y}\subsetneq \Gc^1$. Then there is an element $F\in (\Gc)^*$ so that
            \begin{equation}
                \left\{
                    \begin{array}{ll}
                        F^1 \neq 0 \\
                        F|_{\sum\limits_{y\in\Gc}\ima \adg{y}}=0 & 
                    \end{array}
                \right.
            \end{equation}
            If so, it is clear that
            \begin{equation}
                F([x_i,x_j])=0, \text{ unless } 1\leq i,j\leq n-m.
            \end{equation}
            It follows from Proposition \ref{rankbf} that $k=\dim\Omega_{F}\leq n-m<k$, a contradiction. This proves the first item. 
            
    \vskip0.3cm        
            \item Let $\Gc_\complex=\langle x_1,\dots,x_n\rangle_\complex$ be the complexification of $\Gc$. By Proposition \ref{commutative}, $\{\adg{x_i}:i=1,..,n-m\}$ is a family of commuting endomorphisms of the real type. Hence, there is a basis of $\Gc_{\complex}^1$ so that the matrix of  $\ad{x_i}|_{\Gc_\complex^1}$ with respect to that basis is of upper triangular form for every $1\leq i\leq n-m$ (by Proposition \ref{triangular}). It follows from the first item and Lemma \ref{lemma32} that there is a complex linear combination $x$ of $x_1,x_2,\dots,x_{n-m}$ so that $\ad{x}|_{\Gc_\complex^1}$ is isomorphic. Therefore, by using the same notation $\adg{x_i} (i=1,..,n-m)$ for their matrices with respect to that the basis $\{x_1,\dots,x_n\}$, we easily see that the polynomial
            \[
                \det(\zeta_1\adg{x_1}+\cdots +\zeta_{n-m}\adg{x_{n-m}})
            \]
            of variables $\zeta_1,\dots,\zeta_{n-m}$ does not vanish. Hence, there is $\alpha_1,\dots,\alpha_{n-m}\in\reals$ so that 
            \begin{equation}
                \det(\alpha_1\adg{x_1}+\cdots +\alpha_{n-m}\adg{x_{n-m}}) \neq 0.
            \end{equation}
            This prove the existence of a real linear combination $x'$ of $x_1,x_2,\dots,x_{n-m}$ so that $\adg{x'}$ is nonsingular, as required.
            
    \vskip0.3cm        
            \item It follows from the second item proven above that we may assume (after changing basis if necessary) $\adg{x_1}$ is nonsingular. Hence, for each $j\geq 2$, there are $\alpha_{j1},\alpha_{j2},\dots,\alpha_{jm}\in\reals$ so that
            \begin{equation}
                [x_1,x_j]+\alpha_{j1}[x_1,x_{n-m+1}]+\cdots+\alpha_{jm}[x_1,x_{n}]=0.
            \end{equation}
            Therefore, by basis changing $x'_j=x_j+\sum\limits_{i=1}^m \alpha_{ji} x_{n-m+i}$
            if necessary, we can assume from beginning that $[x_1,x_j]=0$ for all $1\leq j\leq n-m$. With this assumption in hand, we may apply the Jacobi's identity for $(x_1,x_j,x_k)$ to obtain
            	\begin{equation}
                	[x_1,[x_j,x_k]]=0 \text{ for all } 1\leq j, k\leq n-m.
            	\end{equation}
            It follows from the nonsingularity of $\adg{x_1}$ that $[x_j,x_k]=0$ for all $1\leq j,k\leq n-m.$ Hence, $\Gc$ is isomorphic to the semidirect sum of $\Gc^1$ with the commutative subalgebra $\mathcal{L}:= \langle x_1,\dots,x_{n-m}\rangle \cong {\reals}^{n-m}$ of $\Gc$, as required.
    	\end{enumerate}
    \end{proof}

Theorem \ref{theorem1} gains in interest if we realize that, by fixing a basis $\{x_1,x_2,\dots,x_n\}$ of the Lie algebra $\mathcal{L}\oplus_\rho \Gc^1$ so that $\mathcal{L}=\langle x_1,\dots,x_{n-m}\rangle$ and $\Gc^1=\langle x_{n-m+1},\dots,x_n\rangle$, the Kirilov's matrix of any $F\in(\mathcal{L}\oplus_\rho \Gc^1)^*$ is of the following form
\begin{equation*}
    \begin{bmatrix}
        \boldsymbol{0} & M_F\\
        -(M_F)^t & \boldsymbol{0}
    \end{bmatrix}
\end{equation*}
where $M_F$ is some matrix of size $(n-m)\times m$. It turns out that the dimension of each $\Omega_F$ depends on the rank of $M_F$. More precisely, $\dim\Omega_F = 2\rank{M_F}$. 

Besides, it is easily seen that the $i$-th row of $M_F$ is exactly 
\begin{equation*}
    \bigl(F^1\circ \adg{x_i}(x_{n-m+1}), F^1\circ \adg{x_i}(x_{n-m+2}),\dots,F^1\circ \adg{x_i}(x_{n})\bigr),
\end{equation*} 
or the (row) matrix of the form $F^1\circ\adg{x_i}$ on $\Gc^1$ with respect to the basis $\{x_{n-m+1},\dots,x_{n}\}$. Therefore, we can connect the rank of $M_F$, as well as the dimension of $\Omega_F$, with the maximum number of linearly independent matrices in the set 
\begin{equation*}
    \{\adg{x_1},\adg{x_2},\dots,\adg{x_{n-m}}\}.
\end{equation*}
Remark that the vector space spanned by the above set and the vector space spanned by the set $\mathcal{S}:=\{\adg{x}:x\in\mathcal{L}\oplus_\rho\Gc^1\}$ coincide. It follows that
\begin{equation}\label{equ37}
    \dim\Omega_F = 2\rank{M_F}=2\dim F^1(\langle \mathcal{S}\rangle),
\end{equation}
where $\langle \mathcal{S}\rangle$ denotes the vector space spanned by $S$, and $F^1(\langle \mathcal{S}\rangle)$ denotes the vector space $\{F^1\circ f: f\in\langle \mathcal{S}\rangle\}$. In the following theorem, we will use the maximum number mentioned above, which is equal to the dimension of $\langle \mathcal{S}\rangle$, to determine some interest properties of $\Gc$, including the decomposition.
\begin{theorem}\label{theorem2}
	Let $\Gc$ satisfy the hypothesis of Theorem \ref{theorem1} and let $\mathcal{S}$ be the set of adjoint actions of all points in $\Gc$, i.e. $\mathcal{S}=\{\adg{x} : x \in \Gc\}$. Denote by $m$ the dimension of $\Gc^1$ and by $r$ the dimension of the vector space spanned by $\mathcal{S}$. Then the following assertions hold. 
	\begin{enumerate}
		\item $k\leq 2r \leq 2(n - m)$,		
		\item $r = n - m$ if and only if $\Gc$ is indecomposable,
		\item $2r=k$ if and only if every non-zero element in the vector space spanned by $\mathcal{S}$ is an automorphism on the vector space $\Gc^1$.		 
	\end{enumerate}
\end{theorem} 

	\begin{proof}
    	Followed by Theorem \ref{theorem1}, $\Gc$ is a direct sum of two commutative subalgebras $\mathcal{L}$ and $\Gc^1$. By the second assertion in that theorem, we can fix a basis $\{x_1, \ldots , x_n\}$ of $\Gc$ so that 
    	\begin{equation}
    	    \left\{\begin{array}{ll}
    	        \mathcal{L} = & \langle x_1, \ldots, x_{n-m}\rangle   \\
    	        \Gc^1 = & \langle x_{n-m+1}, \ldots, x_n\rangle
    	    \end{array}\right.,
    	\end{equation}
    	and $\adg{x_1}$ is an automorphism on the vector space $\Gc^1$.
	\begin{enumerate}
		\item Since $\Gc$ is spanned by $\{x_1,x_2,\dots,x_n\}$, the real vector space spanned by $\mathcal{S}$, $\langle \mathcal{S}\rangle$, is also spanned by $\{\adg{x_i}:i=1,2,\ldots,n\}$. Moreover, the commutation of $\Gc^1$ implies that $\adg{x_i}=\boldsymbol{0}$ for every $n-m+1\leq i\leq n$. Hence, 
		\begin{equation}\label{eq381}
		    r\leq n-m.
		\end{equation}
		On the other hand, in the light of the equation (\ref{equ37}), we easily see that, for any $F\in \Gc^*$ which does not vanish on $\Gc^1$, 
		\begin{equation}\label{eq391}
		    0\neq \dim \Omega_F = 2\dim F^1(\langle \mathcal{S}\rangle)\leq 2\dim\langle \mathcal{S}\rangle =2r.
		\end{equation}
	    The inequalities (\ref{eq381}) and (\ref{eq391}) establish the formula in the first item.
	    
    \vskip0.3cm		
		\item By Theorem \ref{theorem0}, we conclude that $\Gc$ is decomposable if and only if $\Gc=\mathcal{K}\oplus \reals$ for some sub-algebra $\mathcal{K}\subseteq\Gc$. Hence, $\Gc$ is decomposable if and only if there is an element $x\in \Gc\setminus\Gc^1$ so that it is contained in the center of $\Gc$. Because $x\notin \Gc^1$, we get 
		\begin{equation}
		    x=y+z 
		\end{equation}
		for some $0\neq y\in\mathcal{L}$ and for some $z\in\Gc^1$. Since $x$ belongs to the center of $\Gc$ and $\Gc^1$ is commutative, we obtain 
		\begin{equation}
		    \adg{y}=\boldsymbol{0}.
		\end{equation}
		Therefore, $\Gc$ is decomposable if and only if the dimension of the vector space spanned by 
		\begin{equation*}
		    \{\adg{x_1},\ldots,\adg{x_{n-m}}\}
		\end{equation*}
		is strictly less than $n-m$, or $r<n-m$, which is our assertion.
    
    \vskip0.3cm
		\item In light of the equation (\ref{equ37}), what we need to show is the following statement: $\dim F^1(\langle \mathcal{S}\rangle)=\dim (\langle \mathcal{S}\rangle)$ for every $F\in(\Gc)^*$ with $F^1\neq 0$ if and only if every non-zero element of $\langle \mathcal{S}\rangle$ is an automorphism on the vector space $\Gc^1$. Equivalently, $\dim F^1(\langle \mathcal{S}\rangle)<\dim (\langle \mathcal{S}\rangle)$ for some $F\in(\Gc)^*$ with $F^1\neq 0$ if and only if there is a non-zero element $f$ of $\langle \mathcal{S}\rangle$ which is not an automorphism on the vector space $\Gc^1$.
	
	\vskip0.2cm	
		Indeed, the existence of a non-zero element $f\in\langle \mathcal{S}\rangle$ which is not an automorphism on the vector space $\Gc^1$ is equivalent to the existence of an element $0\neq x\in\Gc^1$ so that $f(x)=0$, hence is equivalent to the existence of an element $F\in(\Gc)^*$ so that 
		\begin{equation}\label{equF1f}
		    \left\{\begin{array}{ll}
		        F^1 \neq \boldsymbol{0},  \\
		        F^1\circ f =\boldsymbol{0}. 
		    \end{array}\right.
		\end{equation}
		If the equation (\ref{equF1f}) holds, then by choosing a basis for $\langle\mathcal{S}\rangle$ which contains $f$, we can see easily that the dimension of $F^1(\langle \mathcal{S}\rangle)$ is strictly less than the dimension of $\langle \mathcal{S}\rangle$.
	
	\vskip0.2cm
    	Conversely, if $\dim F^1(\langle \mathcal{S}\rangle)<\dim (\langle \mathcal{S}\rangle)$, then there is a basis $\{f_1,f_2,\dots,f_r\}$ of $\langle \mathcal{S}\rangle$ so that 
    	\begin{equation}
    	    F^1(\alpha f_1+\cdots+\alpha_rf_r)=\boldsymbol{0}
    	\end{equation} 
    	for some non-zero $(\alpha_1,\dots,\alpha_r)\in\reals^r$. This proves the existence of a non-zero element $f\in\langle \mathcal{S}\rangle$ so that the equation (\ref{equF1f}) holds. The proof is completed.
	\end{enumerate}		
	\end{proof}

To close this subsection, let's present the connection between 2-step MD-algebras with 1-step MD-algebras in the next theorem.
    \begin{theorem}\label{theorem3}
        Let $\Gc$ be an MD$_{k}(n)$-algebra. Then $\dim\Gc^2\leq n-k$. Furthermore, the quotient Lie algebra $\Hc:=\Gc/\Gc^2$ is also an MD$_k$-algebra.
    \end{theorem}

    \begin{proof}
        The assertion of the theorem is obviously true when $\Gc^2$ is trivial. Therefore, we only need to consider the case $\dim \Gc^2=p > 0$. Denote by $m$ the dimension of $\Gc^1$ ($m > p$). Let's fix a basis $\{x_1,x_2,\dots,x_n\}$ of $\Gc$ so that $\Gc^1=\langle x_{n-m+1},\dots,x_n\rangle$ and $\Gc^2=\langle x_{n-p+1}, \dots,x_n\rangle$. Since $\Gc^2$ is an ideal of $\Gc$, $x^*_{n-m+1}([x_i,x_j])=0$ for every $i, j$ with $1\leq i\leq n$ and $n-p+1\leq j\leq n$. Therefore, by setting $F = x^*_{n-m+1}$ we get 
        \begin{equation}\label{eq36}
            B_F  
             =\left[\begin{array}{cl}
                M & \boldsymbol{0}\\
                \boldsymbol{0} & \boldsymbol{0}_{p\times p}
            \end{array}\right]
        \end{equation}
        for some square matrix $M$ of order $n-p$. It follows that the rank of $B_F$ is  at most $n-p$.
        On the other hand, because $\Gc$ is an $\MD{n}{k}$-algebra and $F^1 \neq 0$, we have
        \begin{equation}\label{eq37}
            \dim\Omega_F=k.
        \end{equation}
        It follows from Proposition \ref{rankbf} that
        \begin{equation}
            k=\dim\Omega_F= \rank B_F \leq n-p.
        \end{equation}
        In other words, 
        \begin{equation} \dim\Gc^2\leq n-k.\end{equation} 
        This proves the first part of the theorem.
        \\
        To prove the second part, we  still assume $\{x_1,x_2,\dots, x_n\}$ is a basis of $\Gc$ so that $\{x_{n-m+1},\dots, x_n\}$ and $\{x_{n-p+1},\dots,x_n\} (p<m)$ are bases of $\Gc^1$ and $\Gc^2$, respectively.  It is standard to check that
        \begin{equation}
            x_q^*([x_i,x_j]) = \overline{x_q^*}([x_i+\Gc^2,x_j+\Gc^2]) \text{ for any } 1\leq i,j\leq n, \text{ and } 1\leq q \leq  n-p,
        \end{equation}
        where $\overline{x_q^*}$ is the corresponding element of $x_q^*$ in $(\Hc^1)^*$. Therefore, the dimensions of the coadjoint orbits  $\Omega_{\overline{F}}$ (of $\Hc$) and $\Omega_{F}$ (of $\Gc$) are equal for any $F\in(\Gc)^*$ with $\overline{F}|_{\Hc^1}\neq 0$. This completes the proof.
    \end{proof}

\subsection{The complete classification of all MD\texorpdfstring{$\boldsymbol{_{n-1}(n)}$}{(n-1)(n)}-algebras}

    Now, we will state the last main result of the paper, which gives the complete classification of all indecomposable MD$_{n-1}(n)$-algebras for $n \geq 4$.  

    \begin{theorem}\label{main-theorem}
        Let $\Gc$ be an indecomposable MD$_{n-1}(n)$-algebra with $n\geq 4$. Then $\Gc$ is isomorphic to one of the followings: 
        \begin{enumerate}
            \item The real Heisenberg Lie algebra \[\mathfrak{h}_{2m+1}=\langle x_i,y_i,z:i=1,\dots,m\rangle, \] 
            where non-zero Lie brackets are given by $[x_i,y_i]=z$ for every $1\leq i\leq m$.
                
            \item \label{item7}The Lie algebra \[\mathfrak{s}_{5,45}=\langle x_1,x_2,y_1,y_2,z\rangle,\] where non-zero Lie brackets are given by
                \[
                    [x_1,y_1]=y_1, [x_1,y_2]=y_2, [x_1,z]=2z, [x_2,y_1]=y_2, [x_2,y_2]=-y_1, [y_1,y_2]=z.
                \]
        \end{enumerate}
    \end{theorem}

\begin{proof}
	  The proof will be divided into 3 cases.
	\begin{itemize}
		\item \textbf{Case 1.} Assume $\dim\Gc^1=1$. Followed by the classifications of solvable real Lie algebras having 1-dimensional derived ideal in \cite{Vu-MD}, we easily see that $\Gc$ is isomorphic to $\mathfrak{h}_{2m+1}$ with $2 \leq m \in \mathbb{N}$ because $\Gc$ is indecomposable and $n \geq 4$.
		
	\vskip0.3cm	
		\item \textbf{Case 2.} Assume that $\Gc^1$ is commutative and $\dim\Gc^1\geq 2$. We will show that this case is excluded.
		
		We first claim that $\dim \Gc$ is exactly five. Indeed, according to Theorem \ref{theorem1}, one can see that $\Gc=\mathcal{L}\oplus_{\rho} \Gc^1$ where $\mathcal{L}$ is a commutative subalgebra of $\Gc$ with $\mathcal{L}\cap \Gc^1=\{0\}$. Hence, we can choose a basis $\mathfrak{b}=\left\{x_1,x_2,\dots, x_n\right\}$ of $\Gc$ so that 
		\begin{equation}
    		\left\{\begin{array}{rl}
    		    \Gc^1= & \langle x_{n-m+1},x_{n-m+2},\dots,x_n\rangle   \\
    		     \mathcal{L} = & \langle x_1,x_2,\dots,x_{n-m}\rangle 
    		\end{array}\right..
		\end{equation}
		Because both $\Gc^1$ and $\mathcal{L}$ are commutative, we have
		\begin{equation}\label{eq38}
            [x_i,x_j]=0 \quad \text{if } \left[
                \begin{array}{c}
                    1\leq i,j \leq n-m \\ 
                    n-m+1\leq i,j\leq n
                \end{array}
                \right..
        \end{equation}
		
		On the other hand, the commutation of $\Gc^1$ implies that $\{\adg{x}:x\in \mathcal{L}\}$ is a family of commuting endomorphisms [Proposition \ref{commutative}]. As a consequence of the second item of Proposition \ref{triangular}, we can assume that the matrix of $\adg{x_i}$ with respect to the basis $\mathfrak{b}$ is of block upper triangular form in which each block is of size up to 2, for every $1\leq i\leq n-k$.
		With this assumption in hand, the equality {\rm (\ref{eq38})} implies that
		\begin{equation}
    		x_n^*([x_i,x_{j}]) = 0 \quad \text{if } \left[
                \begin{array}{c}
                    1\leq i,j \leq n-2 \\
                    n-1\leq i,j \leq n \\
                \end{array}
            \right..
		\end{equation}
		Therefore, 
		\begin{equation}
		    0\neq \dim\Omega_{x_n^*}=\rank B_{x_n^*} \leq 4
		\end{equation}
		Hence, $n-1=\dim\Omega_{x_n^*}\leq 4$, or $n\leq 5.$ Since $\dim\Omega_F$ is even for every $F\in\Gc^*$ and $n \geq 4$, $n$ must be equal to $5$.
		
		To the best of our knowledge, all 1-step solvable real Lie algebras of dimension 5 are completely classified by Jacqueline Dozias  \cite{5-dim,Snob}. In particular, Vu L. A. and K. P. Shum gave the classification of 1-step solvable MD-algebras of dimension 5 in \cite{Vu-Shum}. According to their classification, there is no indecomposable 1-step solvable MD-algebra of dimension 5 so that its maximal dimension of $K$-orbits is 4. Therefore, this case is excluded.
    
    \vskip0.3cm
		\item \textbf{Case 3.} Assume that $\Gc^1$ is non-commutative. If so, $\dim\Gc^1\geq 2$ and $\dim\Gc^2\geq 1$. It follows from Theorem \ref{theorem3} that $\dim\Gc^2\leq \dim\Gc-\dim\Omega_F=1$ for every $0\neq F\in (\Gc^1)^*$. Hence, $\dim\Gc^2=1$ and  $\Hc=\Gc/\Gc^2$ is a solvable real Lie algebra in which $\dim\Hc=n-1=\dim\Omega_F$ for every $0\neq F\in (\Gc^1)^*$. In other word, $\Hc$ is an indecomposable SMD-algebra. It follows from \cite[Theorem 1]{Son-Viet} that $n - 1 = \dim\Hc\leq 4$. Hence, $n \leq 5$. Once again, by using the classification of 5-dimensional Lie algebras in \cite{5-dim,Snob}, especially of those in which the first derived ideal is non-commutative in \cite{Vu-MD}, we get exactly one 5-dimensional Lie algebra satisfying the requirement which is defined in item \ref{item7}. Remark that this algebra is denoted as $\mathfrak{s}_{5,45}$ in \cite{Snob} and the proof is complete.
	\end{itemize}
\end{proof}

\begin{remark} To sharpen Theorem \ref{main-theorem}, using Theorem \ref{theorem0}, we complete the paper by the following remarks.
	\begin{enumerate}
		\item[(1)] 
		All solvable Lie algebras of dimension 3 are obviously MD-algebras. Therefore, except $\reals^3$ (the real 3-dimensional abelian Lie algebra), all the remaining solvable non-abelian ones are MD$_2(3)$-algebras. One can find their classification in \cite{Snob}.
		\item[(2)] 
		Any decomposable MD$_{n-1}(n)$-algebra is a direct sum of $\reals$ and an indecomposable non-abelian SMD-algebra of dimension $n-1$. Therefore, by applying directly the classification of SMD-algebras in \cite[Theorem 1]{Son-Viet}, we easily see that any decomposable MD$_{n-1}(n)$-algebra is isomorphic to either $\textnormal{aff}(\reals) \oplus \reals$ for $n = 3$ or $\textnormal{aff}(\complex) \oplus \reals$ for $n = 5$ where 
		\begin{itemize}
			\item 
			$\textnormal{aff}(\reals) = \langle x,y\rangle$ is defined by $[x, y] = y$;
			\item 
			$\textnormal{aff}(\complex)= \langle x_1,x_2,y_1,y_2\rangle$ is defined by $[x_1,x_2]=[y_1,y_2]=0$, 
			$[x_1,y_1]=y_1, [x_1,y_2]=y_2$, and  $[x_2,y_1]=y_2,[x_2,y_2]=-y_1.$
		\end{itemize}
	\end{enumerate}	        
\end{remark} 

\bigskip
{\bf Acknowledgment.} The first author would like to express his deep gratitude to his wife - Thuy Nguyen Thanh for her encouragement and support.



\begin{thebibliography}{1}
\bibitem{Arnal}
D.~Arnal, M.~Cahen and J.~Ludwig.
\newblock  Lie Groups whose Coadjoint Orbits are of Dimension Smaller or Equal to Two.
\newblock  {\em Letters in Mathematical Physics}, 33:183--186, 1995.

\bibitem{BDL09}
 G.~R.~Belitskii, A.~R.~Dmytryshyn and R.~Lipyanski.
\newblock  Problems of classifying associative or Lie algebras over a field of characteristic not two and finite metabelian groups are wild.
\newblock  {\em Electronic Journal of Linear Algebra}, 18:516--529, 2009.

\bibitem{Con82}
 A.~Connes.
\newblock  A survey of Foliations and Operator Algebras.
\newblock  {\em Proceedings of Symposia in Pure Mathematics}, 38(1):512--628, 1982.

\bibitem{Do99}
Do Ngoc Diep.
\newblock  {\em Method of Noncommutative Geometry for Group $C^*$-algebras}.
\newblock  Chapman and Hall-CRC Press, Cambridge, 1999.

\bibitem{FKPS18}
V.~Futorny, T.~Klymchuk, A.~P.~Petravchuk, V.~V.~Sergeichuk. 
\newblock Wildness of the problems of classifying two-dimensional spaces of commuting linear operators and certain Lie algebras. 
\newblock {\em Linear Algebra and its Applications}, 536:201--209, 2018.

\bibitem{FKPS18-2}
G.~Belitskii, A.~R.~Dmytryshyn, R.~Lipyanski, V.~V.~Sergeichuk, A.~Tsurkov. 
\newblock Problems of classifying associative or Lie algebras over a field of characteristic not two and finite metabelian groups are wild. 
\newblock {\em Electronic Journal of Linear Algebra}, 18:516–-529, 2009.

\bibitem{FKPS18-3}
G.~Belitskii, R.~Lipyanski, V.~V.~Sergeichuk, A.~Tsurkov. 
\newblock Problems of classifying associative or Lie algebras and triples of symmetric or skew-symmetric matrices are wild. 
\newblock {\em Linear Algebra and its Applications}, 407:249–-262, 2005.

\bibitem{FKPS18-4}
P.~A.~Brooksbank, J.~Maglione, J.~B.~Wilson. 
\newblock A fast isomorphism test for groups whose Lie algebra has genus 2. 
\newblock {\em Journal of Algebra}, 473:545–-590, 2017.

\bibitem{5-dim}
Jacqueline Dozias.
\newblock  {\em Sur les alg\`ebres de Lie r\'esolubles, r\'eelles, de dimension inf\'erieure ou \'egale \`a 5}.
\newblock  Th\`ese de $3^e$ Cycle, Falcult\'e des Sciences de Paris, 1963.

\bibitem{GR19}
 Michel Goze and Elisabeth Remm.
\newblock  Coadjoint Orbits of Lie Algebras and Cartan Class.
\newblock  {\em Symmetry, Integrability and Geometry: Methods and Applications (SIGMA)}, 15(2):1--20, 2019.

\bibitem{Heydar}
Heydar Radjavi and Peter Rosenthal.
\newblock  {\em Simultaneous Triangularization}.
\newblock  Springer, 2000.

\bibitem{Kiri}
Alexandre AleksandrovichKirillov.
\newblock  {\em Elements of the Theory of Representations}.
\newblock  Springer Verlag, Berlin - Heidenberg - New York, 1976.

\bibitem{KN87}
 Alexandre A.~Kirillov and Yuri A.~Neretin.
\newblock  The Veriety $\mathcal{A}_n$ of $n$-Dimensional Lie Algebra Structures.
\newblock  {\em American Mathematical Society Translations}, 137(2):21--30, 1987.

\bibitem{Le90-2}
 Le Anh Vu.
\newblock  On the Foliations Formed by the Generic $K$-orbits of the MD4-Groups.
\newblock  {\em Acta Mathematica Vietnamica}, 15(2):39--55, 1990.

\bibitem{Le93} Le Anh Vu.
\newblock On the foliations by the generic $K$-orbits of a class  of  solvable Lie groups.
\newblock {\em Vestnik of Moscow State University, ser. 1, Mat.Mekh}. (1993), No 3, 26--29 (In Russian).

\bibitem{Boza13} Luis Boza, Eugenio M.~Fedriani, JJuan N\'u\~nez, and \'Angel F. Tenorio.
\newblock A historical Review of the Classifications of Lie Algebras.
\newblock {\em Revista de la Uni\'on Matem\'atica Argentina}. 54(2):75--99, 2013.

\bibitem{LHT11}
 Vu Le Anh, Hieu Ha Van and Nghia Tran Thi Hieu.
\newblock  Classification of 5-dimensional MD-algebras having non-commutative derived ideals.
\newblock  {\em East-West Journal of Mathematics}, 13(2):115--129, 2011.

\bibitem{Morris}
 Morris Newman.
\newblock  Two Classical Theorems on Commuting Matrices.
\newblock  {\em Journal of research of the National Bureau of Standards - B. Mathematics and Mathematical Physics}, 71B:69--71, 1967.

\bibitem{Snob}
Libor \v{S}nob and Pavel Winternitz.
\newblock  {\em Classification and Identification of Lie algebras}.
\newblock  American Mathematical Society, 2014.

\bibitem{Son-Viet}
 Vuong Manh Son and Ho Huu Viet.
\newblock  Sur la Struture Des $C^*$-Alg\`ebres D'une Classe de Groupes de Lie.
\newblock  {\em Journal of Operator Theory}, 11:77--90, 1984.

\bibitem{Vu-MD}
 L.~A.~Vu, H.~V.~Hieu, N.~A.~Tuan, C.~T.~T.~Hai, and N.~T.~M.~Tuyen.
\newblock  Classification of Real solvable Lie algebras whose simply connected Lie groups have only zero or maximal dimensional Coadjoint Orbits.
\newblock  {\em Revista de la Uni\'on Matem\'atica Argentina}, 57(2):119--143, 2016.

\bibitem{Vu-Shum}
 Le Anh Vu and Kar Ping Shum.
\newblock  Classification of 5-dimensional MD-algebras having Commutative Derived Ideals.
\newblock  {\em Advances in Algebra and Combinatorics}, 353--371, 2008.

\end{thebibliography}
\end{document}